\documentclass[a4paper,12pt]{article}
\usepackage{hyperref}
\usepackage{cmap}                        
\usepackage[cp1251]{inputenc}            
\usepackage[russian,english]{babel}
\usepackage{amssymb,amsmath, amsthm, amscd,ifthen}
\usepackage{graphicx}

\newcommand{\R}{{\mathbb R}}
\newcommand{\N}{{\mathbb N}}

\newtheorem{prob}{Problem}

\newtheorem{thm}{Theorem}
\newtheorem{lem}{Lemma}
\newtheorem{cor}{Corollary}
\date{}
\newtheorem{obs}{Observation}

\DeclareMathOperator{\aff}{aff}
\DeclareMathOperator{\lin}{lin}

\title{Number of double-normal pairs in space.}
\author{Andrey Kupavskii\footnote{Moscow Institute of Physics and Technology, \'Ecole Polytechnique F\'ed\'erale de Lausanne; Email: {\tt kupavskii@yandex.ru} \ \ Research supported in part by the Swiss National Science Foundation Grants 200021-137574 and 200020-14453 and by the grant N 15-01-03530 of the Russian Foundation for Basic Research.}}
\begin{document}

\maketitle
\begin{abstract}Given a set $V$ of points in $\R^d$, two points $p$, $q$ from $V$ form a double-normal pair, if the set $V$ lies between two parallel hyperplanes that pass through $p$ and $q$, respectively, and that are orthogonal to the segment $pq$. In this paper we study the maximum number $N_d(n)$ of double-normal pairs in a set of $n$ points in $\R^d$. It is not difficult to get from the famous Erd\H os-Stone theorem that $N_d(n) = \frac 12(1-1/k)n^2+o(n^2)$ for a suitable integer $k = k(d)$ and it was shown in the paper by J. Pach and K. Swanepoel that $\lceil d/2\rceil\le k(d)\le d-1$ and that asymptotically $k(d)\gtrsim d-O(\log d)$.

In this paper we sharpen the upper bound on $k(d)$, which, in particular, gives $k(4)=2$ and $k(5)=3$ in addition to the equality $k(3)=2$ established by J. Pach and K. Swanepoel. Asymptotically we get $k(d)\le d- \log_2k(d) = d - (1+ o(1)) \log_2k(d)$ and show that this problem is connected with the problem of determining the maximum number of points in $\R^d$ that form pairwise acute (or non-obtuse) angles.
\end{abstract}

\section{Introduction}
Let $V$ be a set of $n$ points in $\R^d$. We say that two points $p$, $q$ from $V$ form a \textit{double-normal pair}, if the set $V$ lies between two parallel hyperplanes $H_p, H_q$ that pass through $p$ and $q$, respectively, and that are orthogonal to the segment $pq$. We call a double-normal pair $p,q$ \textit{strict}, if $V\cap(H_p\cup H_q) = \{p,q\}$. Following this definition, we may consider the (strict) double-normal graph on $V$. We say that $G$ is the \textit{(strict) double-normal graph of $V$} if $V$ is the vertex set of $G$ and two points from $V$ are connected by an edge iff they form a (strict) double-normal pair.

In 1957, Erd\H os \cite{Ea} conjectured that if any pair of points in $V\subset\R^d$ forms a double-normal pair, then $|V|\le 2^d$. This problem is better known in another, equivalent form: given a set $V\subset \R^d$ such that no three points from $V$ form an obtuse angle, is it true that $|V|\le 2^d$? This was proved by L. Danzer and B. Gr\" unbaum in \cite{DG}. L. Danzer and B. Gr\" unbaum in their paper asked the same question with a slight modification: all the triples in the set must form acute angles (instead of non-obtuse). In our terms, this question is the question of Erd\H os with double-normals changed to \textit{strict} double-normals. L. Danzer and B. Gr\" unbaum conjectured that the size of $V$, in which every two vertices form a strict double-normal pair, cannot be bigger than $2d-1$. It was believed in for a long time, till Erd\H os and F\"uredi in \cite{EF} disproved it, showing an unexpectedly different bound, which states that the size of such $V$  may be exponential in the dimension.

Formulated in double-normal graph terms, Erd\H os asked, how big a set $V\subset \R^d$ may be, if its double-normal graph is a complete graph. Or, what is the largest size $N(d)$ of a complete graph that may be represented as a double-normal graph in $\R^d$. Similarly, we can ask what is the largest size $N'(d)$ of a complete graph that may be represented as a \textit{strict} double-normal graph in $\R^d$. In general, it is a classical question of extremal graph theory to ask what are the possible values of a clique number in a graph belonging to a certain class of graphs. In the paper \cite{MS} Martini and Soltan asked another classical extremal-graph-theory-type question for double normal graphs: what are the maximum numbers $N_d(n)$ and $N'_d(n)$ of double-normal pairs and strict double-normal pairs formed by pairs of points from a set among all sets of $n$ points in $\R^d$?  In terms of double-normal graphs this question is formulated as follows: what is the maximum number of edges in a (strict) double-normal graph on $n$ vertices in $\R^d$?

It follows from the famous Erd\H os-Stone theorem that $N_d(n) = \frac 12(1-1/k)n^2+o(n^2)$ for a suitable integer $k = k(d)$, where $k$ is the largest integer such that for any $r\in \N$ we can find a double-normal graph in $\R^d$ that has a complete $k$-partite graph $K_k(r)$ with sizes of parts equal to $r$ as a subgraph. Analogously, $N'_d(n) = \frac 12(1-1/k')n^2+o(n^2)$ for a suitable integer $k' = k'(d)$. Note  that $N_d(n)\ge N'_d(n)$ and $k\ge k'$. Thus, to find the asymptotic of $N_d(n)$  in case $N_d(n)$ being quadratic all we need is to determine $k$. The same takes place for $N'_d(n)$.

In the case of the plane and a two-dimensional sphere, however, these functions are linear in $n$. This was shown by J. Pach and K. Swanepoel in \cite{PS2}. They, in fact, found exact values for the planar functions and their spherical analogues. In particular, they showed that $N'_2(n) = n$ and $N_2(n) = 3\lfloor n/2\rfloor.$

In a complementary paper \cite{PS1} they showed that already in $\R^3$ the functions are quadratic. More specifically, they showed that $k(3) =k'(3) =2$. In general, they proved the following

\begin{thm}[\cite{PS1}]\label{thps}
We have $\lceil d/2\rceil\le k'(d)\le k(d)\le d-1$ and asymptotically $k(d)\ge k'(d)\gtrsim d-O(\log d)$.
\end{thm}

Although Theorem \ref{thps} gives $k(3)=k'(3)=2$, it does not allow to decide whether $k(4)$ or $k'(4)$ equals 2 or 3. In this paper we improve both the upper and the lower bounds from Theorem \ref{thps}, which gives, in particular, $k(4)=k'(4)=2$. While the improvement of the upper bound is easy to state, it is more difficult with the lower bound, since it it stated in Theorem \ref{thps} somewhat imprecisely.  Thus, we first formulate a more accurate version of the asymptotical lower bound:

\begin{thm}[\cite{PS1}]\label{thps2} Let $m\ge 2$. Suppose that there exist $m$ points $p_1,\ldots, p_m \in \R^d$
and $m$ unit vectors $u_1,\ldots, u_m \in \R^d$ such that, for all triples of distinct $i, j, k$,
the angle $\angle p_ip_jp_k$ is acute, and
$$ \langle u_i, p_i-p_j\rangle < \langle u_i, p_k-p_j\rangle < \langle u_i, p_j-p_i\rangle. $$
Then, for any $N \in \N$, there exists a strict double-normal graph in $\R^{d+m}$
containing a complete m-partite $K_m(N)$. In particular, $k'(d + m) \ge m$.
\end{thm}
Denote by $D''(m)$ the minimal dimension $d$ in which points and vectors with conditions described in Theorem \ref{thps2} exist. Then it is possible to state the result of Theorem \ref{thps2} in the following form: $k'(d+m)+D''(k'(d+m))\ge d+m$ or, changing the notation,
\begin{equation}\label{e00} k'(d)+D''(k'(d))\ge d.\end{equation}
We are almost ready to state both new lower and upper bounds. But before that we need to define two functions. Let $D(m)$ be the minimal dimension $d$ in which $m$ points forming only non-obtuse angles exist. Let $D'(m)$ be the analogous function for acute angles. We have $D(N(d))=d, D'(N'(d))=d$. It is also clear that $D(m)\le D'(m)\le D''(m)$. From the result of Danzer and Gr\" unbaum we know that $D(m) = \lceil \log_2 m\rceil$, while from the result of V. Harangi \cite{H}, who improved the bound due to Erd\H os and F\" uredi to $N'(d)\ge 1.2^d$ for large $d$, we have $D'(m)\le \log_{1.2} m$ for large $m$.

\begin{thm}\label{thmain} We have
\begin{align} \label{e01} k(d)+D(k(d))&\le d,\\ \label{e02} k'(d)+D'(k'(d))&\ge d.\end{align}

Numerically, we have $k(d)+\log_2 k(d)\le d$ for all $d$ and $k'(d)+\log_{1.2} k'(d)\ge d$ for large $d$.
\end{thm}

The upper bound from Theorem \ref{thmain} together with the lower bound from Theorem \ref{thps} allows us to state the following corollary, which extends the result $k(3)=k'(3)=2$ due to Pach and Swanepoel:

\begin{cor} We have $k'(4) = k(4)=2, k'(5) =  k(5)=3, k'(7) = k(7) =4$.
\end{cor}
Thus, the next value to be determined is $k(6)$. In general, we ask if the bound (\ref{e02}) is sharp:
\begin{prob}\label{co1} Is it true that we have $k'(d)+D'(k'(d))= d$ and, moreover, $k(d)+D'(k(d))= d$?
\end{prob}
It would follow from a positive answer to Problem \ref{co1} that $k(d)=k'(d)$ for all $d$. We formulate this as a separate problem:
\begin{prob}\label{co2} Is it true that we have $k'(d)= k(d)$ for all $d$?
\end{prob}

In Section 2 we present the proof of the upper bound from Theorem \ref{thmain}. In Section 3 we give the proof of the lower bound from Theorem \ref{thmain}.\\

We say that a graph $G$ has a (strict) double-normal representation in $\R^d$, if there is a (strict) double-normal graph in $\R^d$ isomorphic to $G$.
 Instead of double-normal representations and strict double-normal representations  one can consider \textit{almost double-normal} representations. We say that a graph $G$ admits an almost double-normal representation in $\R^d$, if for any $\delta >0$ the graph may be represented in $\R^d$ by a set $V$ of points, with two points $x,y$ connected by an edge if for any $z\in V$ the angles $\angle xyz, \angle yxz \le \frac {\pi}2+\delta$.

Denote by $k_1(d)$ the analogue of $k,k'$ for this type of representations. In fact, slightly modifying the proof of both bounds in Theorem \ref{thmain}, one can prove
\begin{thm}\label{thmalmost} We have
$$k_1(d)+D(k_1(d)) = d, \ \ \text{that is, }\ k_1(d)+\lceil \log_2 k_1(d)\rceil = d.$$
\end{thm}
Thus, we can determine the exact value of the $k$-function in this somewhat modified scenario. We leave it to the reader to check that the proof of theorem \ref{thmain} works for this situation.\\

Throughout the paper we use the following notations. The scalar product of two vectors $u,v$ is denoted by $\langle u,v \rangle$. The linear span of a set of vectors $U$ we denote $\lin U$, the affine span of a set of points $S$ we denote $\aff S$. By $\dim W$ we denote the dimension of a subspace or a plane $W$. The angle formed by points $x,y,z$ with sides $yx, yz$ and its angular measure we denote by $\angle xyz$. Throughout the paper all angles have angular measure from $(0,\pi)$. The Euclidean distance between two points $x,y$ we denote $\|x-y\|$.

\section{Upper bound on $\mathbf{k(d)}$}

We follow the steps of the proof from \cite{PS1} used to obtain the upper bound on $k(d)$. The goal is to prove a strengthened version of Proposition 8 from \cite{PS1}, which is formulated in the forthcoming theorem. Despite that some parts of the proof of the theorem are identical to the ones from \cite {PS1}, we will present an almost complete proof of the theorem for the sake of clarity of exposition. In what follows we do not specify the parts that are taken from \cite{PS1}, but after the proof we are going to comment on what is new and what is old in the statement and in the proof.

\begin{thm}\label{thbase} Put $k=k(d)$. There exists a family of (not necessarily distinct) vectors $u_{ij}$, where $i,j=1,\ldots,k$ so that the following conditions hold:
 \begin{align}
  &\label{e1}U_1:=\{u_{11}, \ldots, u_{kk}\}\text{ form an orthogonal set of vectors.}\\
  &\label{e2}u_{ii}\text{ is orthogonal to } u_{jl}\text{ for any }i,j,l=1,\ldots, k,\text{ such that }j\ne l.\\
  &\label{e3}\dim\lin U_2\ge \log_2k\text{ for } U_2 := \{u_{sl}: s,l=1\ldots, k, s\neq l\}. \end{align}
\end{thm}

\textbf{Remark. } Bound (\ref{e01}) from Theorem \ref{thmain} follows easily from Theorem \ref{thbase} by the following observation: $d\ge \dim \lin(U_1\cup U_2)= \dim \lin U_1+ \dim\lin U_2\ge k+\lceil \log_2 k\rceil = k+D(k)$.

The rest of the section is dedicated to the proof of Theorem \ref{thbase}.
In what follows, we will use the following notation: $f(\epsilon)\gg g(\epsilon)$, if $\lim_{\epsilon \to 0} g(\epsilon)/f(\epsilon)\to 0$, in addition to the traditional $(o, O, \Omega)$-notation with respect to the limit $\epsilon \to 0$. Also, we will say that two vectors are $\epsilon$-orthogonal ($\epsilon$-parallel), if the angle between them is equal to $\pi/2+f(\epsilon)$ ($f(\epsilon)$), where $f(\epsilon)= o(1)$.

First we give a sketch of the proof.

\textbf{Step 1. } We show that for any $\epsilon>0$ there exists $N$ such that for any complete $k$-partite graph $K_k(N)$ with parts $V_1,\ldots, V_k$ of size $N$ contained in a double-normal graph and for any $i$ we can choose points $A_i, a_i, b_i,C_i$ out of $V_i$ such that the following conditions hold:
\begin{align}
  &\label{e4}\text{ Points from }\{A_i, a_i, b_i, C_i\} \text{ satisfy } \angle A_ia_ib_i\ge \pi -\epsilon, \angle a_ib_iC_i\ge \pi-\epsilon,\\
  &\label{e5}\|A_{i+1}-C_{i+1}\|\le \epsilon \|A_i-C_i\|, \\
  &\label{e6}\|A_{i}-b_{i}\|\ge 1/2 \|A_i-C_i\|,\\
  &\label{e7}\|a_{i}-b_{i}\|\le \epsilon \|A_i-b_i\|. 
\end{align}

\textbf{Step 2. } For the points chosen in Step 1 and sufficiently small $\epsilon>0$, set $u_{ii} = \|A_i-b_i\|^{-1}(A_i-b_i)$ and $u_{ij} = \|b_j-b_i\|^{-1}(b_j-b_i)$.
We verify the properties
\begin{align}
&\label{e8}|\langle u_{ii},u_{ij}\rangle|\le \epsilon,  i,j = 1,\ldots, k, i\neq j,\\
&\label{e9} |\langle u_{ii},u_{jj}\rangle|\le 4\epsilon,\ \ i,j = 1,\ldots, k, i\neq j,\\
&\label{e10} |\langle u_{ii},u_{jl}\rangle| = o(1),\ \ i,j,l = 1,\ldots, k, i\neq j\neq l, i\neq l.
\end{align}

\textbf{Step 3. } We set $\epsilon=1/n$ and take a subsequence of $n$'s such that each $u_{ij}^{(n)}$ converge as $n \to \infty$. The limiting set of vectors is the desired one.

\textbf{Step 4. } We give the estimate on the dimension of $\lin U_2$. \\

\textbf{Remark.} Note that we have the same notation for the vectors for a fixed $\epsilon$ and the limiting vectors, which, however, should not cause confusion since it should be clear from the context which out of the two we are dealing with. In Step 1, Step 2 we deal only with $u_{ij}$ for a fixed $\epsilon$.\\

\textbf{Step 1. } First, choose some $0<\epsilon<1$ and take a sufficiently large $N$ depending on $\epsilon$ and $d$ and apply [Theorem 2.6, \cite{P}] to each part $V_i$ of the complete multipartite graph $K_k(N)$ contained in a double-normal graph. We get that any $V_i$ contains a subset $V'_i$ of points $\{p_i^1,\ldots, p_i^{N'}\}$ such that for any $i<j<l, i,j,l = 1,\ldots, N'$ we have $\angle p_ip_jp_l\ge \pi-\epsilon$, where $N' = (2t+3)t^{k-1}$ and $t = \lceil (\epsilon\cos\epsilon)^{-1}\rceil $. We replace the original $V_i$ by $V_i'$.

Next, we modify the sets $V_i$ so that condition (\ref{e5}) holds. We do it in $k$ steps. At the first step we rename the parts in such a way that $V_1$ has the largest diameter over all $V_i$. At step $j$ we do the following to each $V_i, i=j+1,\ldots, k$. Take $V_i, i>j$ and project all the points of $V_i$ onto the segment that connects the furthest points of $V_i$ and split it into $t$ parts of equal length. By the pigeon-hole principle, there is a part that has at least $\frac 1t$-th fraction of all the points in $V_i$. Denote it by $V'_i$. The diameter of $V'_i$ is at most $\frac 1{t\cos\epsilon}$ multiplied by the diameter $V_i$, where $\frac 1{t\cos\epsilon}\le \epsilon.$ The inequality on the diameter $V'_i$ holds since any segment connecting points of $V_i$ form an angle at most $\epsilon$ with the diameter. After that we change $V_i$ for $V'_i$ and rename them in such a way that $V_{j+1}$ has the largest diameter over $V_i, i>j$.

Note that at the end of this procedure  $V_i$ has at least $N't^{-i+1}\ge 2t+ 3$ points. As $A_i$ and $C_i$ we choose the points that form a diameter of $V_i$. Thus, the condition (\ref{e5}) is satisfied.

To obtain (\ref{e7}) for some vertices in $V_i$, we again project the points of $V_i$ on $A_iC_i$ and split the segment into parts of length $\frac 1{2t}\|A_i-C_i\|$. By the pigeon-hole principle, there will be a part that contains at least two (projected) points apart from $A_i$ and $C_i$. Denote them by $b_i$ and $a_i$. We have $\|b_i-a_i\|\le \frac 1{2t}\|A_i-C_i\|$ and either $\min\{\|A_i-b_i\|, \|A_i-a_i\|\}\ge \frac 1{2}\|A_i-C_i\|$ or $\min\{\|C_i-b_i\|, \|C_i-a_i\|\}\ge \frac 1{2}\|A_i-C_i\|.$ Rename the points in such a way that they appear on $A_iC_i$ in the order $A_i,a_i,b_i,C_i$ and $\|A_i-b_i\|\ge \frac 1{2}\|A_i-C_i\|$. Now consider the preimages of $a_i,b_i$ before the projection. We also denote them by $a_i,b_i$. We have $\|a_i-b_i\|\le \frac 1{t\cos\epsilon}\|A_i-b_i\|\le \epsilon\|A_i-b_i\|$. Therefore, (\ref{e7}) is verified for $A_i,a_i, b_i,C_i.$ Moreover, (\ref{e6}) is verified as well. This completes Step 1.\\

\textbf{Step 2. }

The proof of (\ref{e8}) and (\ref{e9}) is identical to the one presented in \cite{PS1} and can be losslessly separated from the rest of the proof. Thus, we omit it, only mentioning that it uses properties (\ref{e4}), (\ref{e5}), (\ref{e6}) of the configuration.

The rest of Step 2 is devoted to verifying (\ref{e10}). We need to prove that for any pairwise distinct $i, j, l=1,\ldots, k$ the vectors $A_i-b_i$ and $b_l-b_j$ are $\epsilon$-orthogonal. For simplicity put $i=1, j=2, l=3$. Consider the two two-dimensional planes $\gamma_a = \aff\{b_2,b_1, a_1\}$ and $\gamma_b = \aff\{b_3,b_1, a_1\}$.

\begin{obs} Consider the projections $b_2^{\gamma}, b_3^{\gamma}$ of points $b_2,b_3$ on the planes $\gamma_b, \gamma_a$ respectively. Then either $b_2^{\gamma}$ lies to the same side from the line $b_1b_3$ as $a_1$  or $b_3^{\gamma}$ lies to the same side from the line $b_1b_2$ as $a_1$.
\end{obs}
\begin{proof} Consider projections $b_2', b_3'$ of $b_2,b_3$ on the line $a_1b_1$. Since $\angle a_1b_1b_i\le \pi/2$, $i=2,3$, both $b_2',b_3'$ lie in the same closed halfline emanating from $b_1$ as $a_1$. Suppose that $b_3'$ is at least as far away from $b_1$ as $b_2'$. We then claim that $b_3^{\gamma}$ lies to the same side from the line $b_1b_2$ as $a_1$.

The point $b_3^{\gamma}$ lies on a line $l$ that is orthogonal to $a_1b_1$ and that passes through $b_3'$. We also know that $b_3^{\gamma}$ must lie between the two lines that are orthogonal to $b_2b_1$ and that pass through the points $b_2$ and $b_1$. Otherwise, the pair $b_1,b_2$ is not double-normal. Then it is easy to see that the part of $l$ that is between the two lines does not cross $b_2b_1$  and thus lies in the same closed halfplane as $a_1$.
\end{proof}


Consider a unit vector $v_a$ that is lying in the plane $\gamma_a$, orthogonal to $b_2b_1$ and pointing to the side of $b_1b_2$ on
which $a_1$ lies. We remark that $v_a$ is $\epsilon$-parallel to $u_{11}$. In what follows, we w.l.o.g. assume that $b_3^{\gamma}$ lies to the same side from the line $b_1b_2$ as $a_1$. Or, in other words, that $\langle u_{23}, v_a\rangle\ge 0$. Consider the projection $b_3^{\gamma}$ of  $b_3$ on the plane $\gamma_a$. We are going to analyze possible locations of $b_3^{\gamma}$ using the fact that almost all pairs of points from $\{b_1,b_2,b_3,a_1,A_1\}$ form a double normal (except for the pairs formed by $\{a_1,A_1,b_1\}$). We denote the projection of $A_1$ on $\gamma_a$ by $A_1^{\gamma}$. See Figure 1, that depicts the plane $\gamma_a$. Note that all the points on the figure indeed lie in that plane, while, in general, $A_1$ does not lie in  $\gamma_a$.

\begin{center}  \includegraphics[width=80mm]{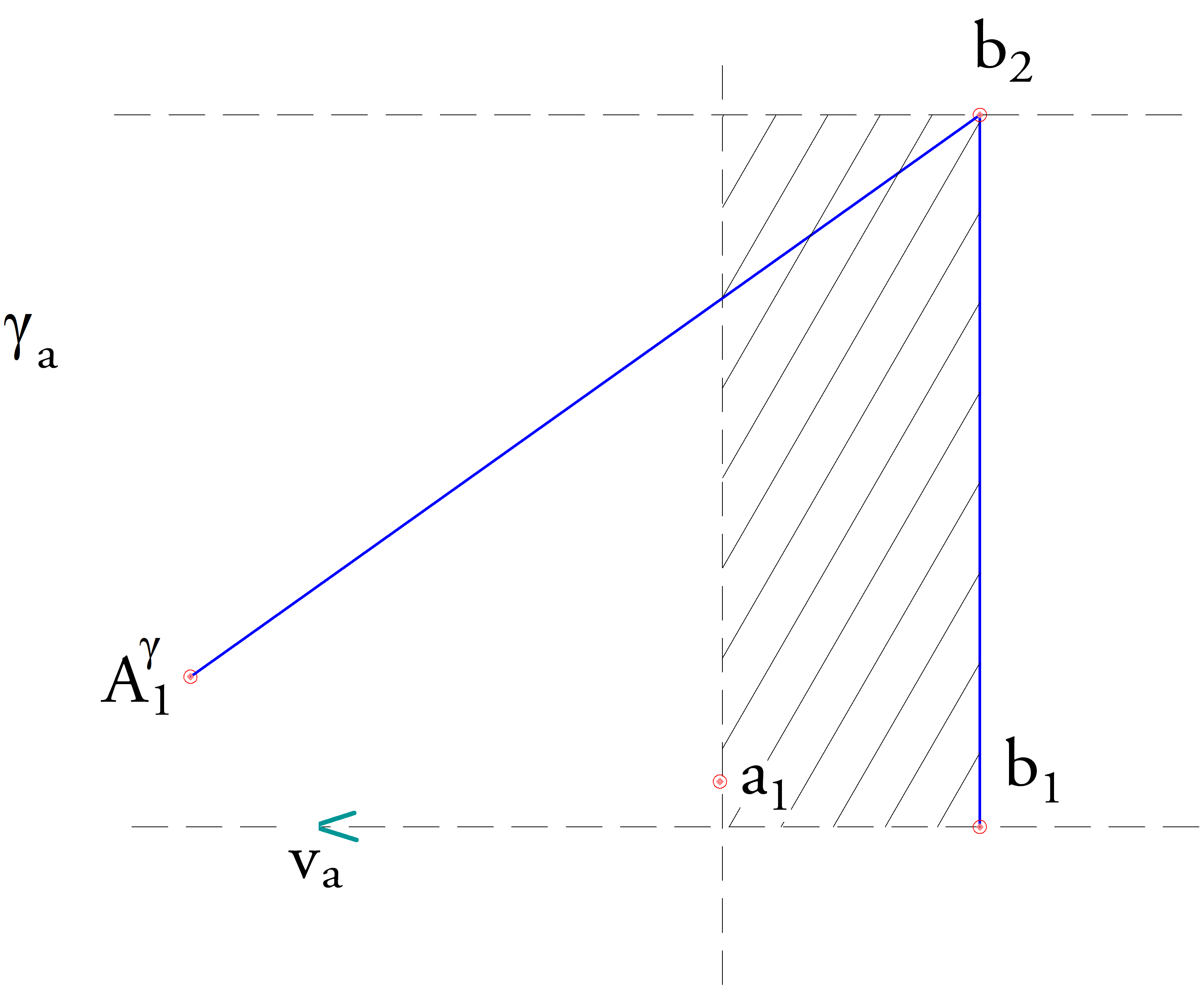}\label{fig1}  \end{center}\begin{center}  Figure 1 \end{center}

First, as we have already mentioned, since $b_1b_2$ is a double-normal pair, $b_3^{\gamma}$ lies in a strip bounded by the two lines that are orthogonal to $b_1b_2$ and that run through $b_1$ and $b_2$ respectively (see Figure \ref{fig1}). Second, $b_3^{\gamma}$ lies in a strip bounded by $b_1b_2$ itself and a line that is parallel to $b_1b_2$ and that passes through $a_1$. One part of this condition holds since $\langle u_{23}, v_a\rangle\ge 0$, while the other one holds since otherwise one of the angles $\angle b_3a_1b_2$ or $\angle b_3a_1b_1$ would be greater than $\pi/2$, which, in turn, gives that $b_3a_1$ is not a double-normal pair. Thus, $b^{\gamma}_3$ falls into the hatched rectangle.

For simplicity, we rescale the configuration so that $\|b_2-b_1\|=1.$ Next, we consider two cases. \\

\textbf{Case 1} The segment $A_1b_1$ is ``not much bigger'' than $b_2b_1$:  $\|A_{1}-b_{1}\| =O(\epsilon^{-1/3}).$

\begin{center}  \includegraphics[width=80mm]{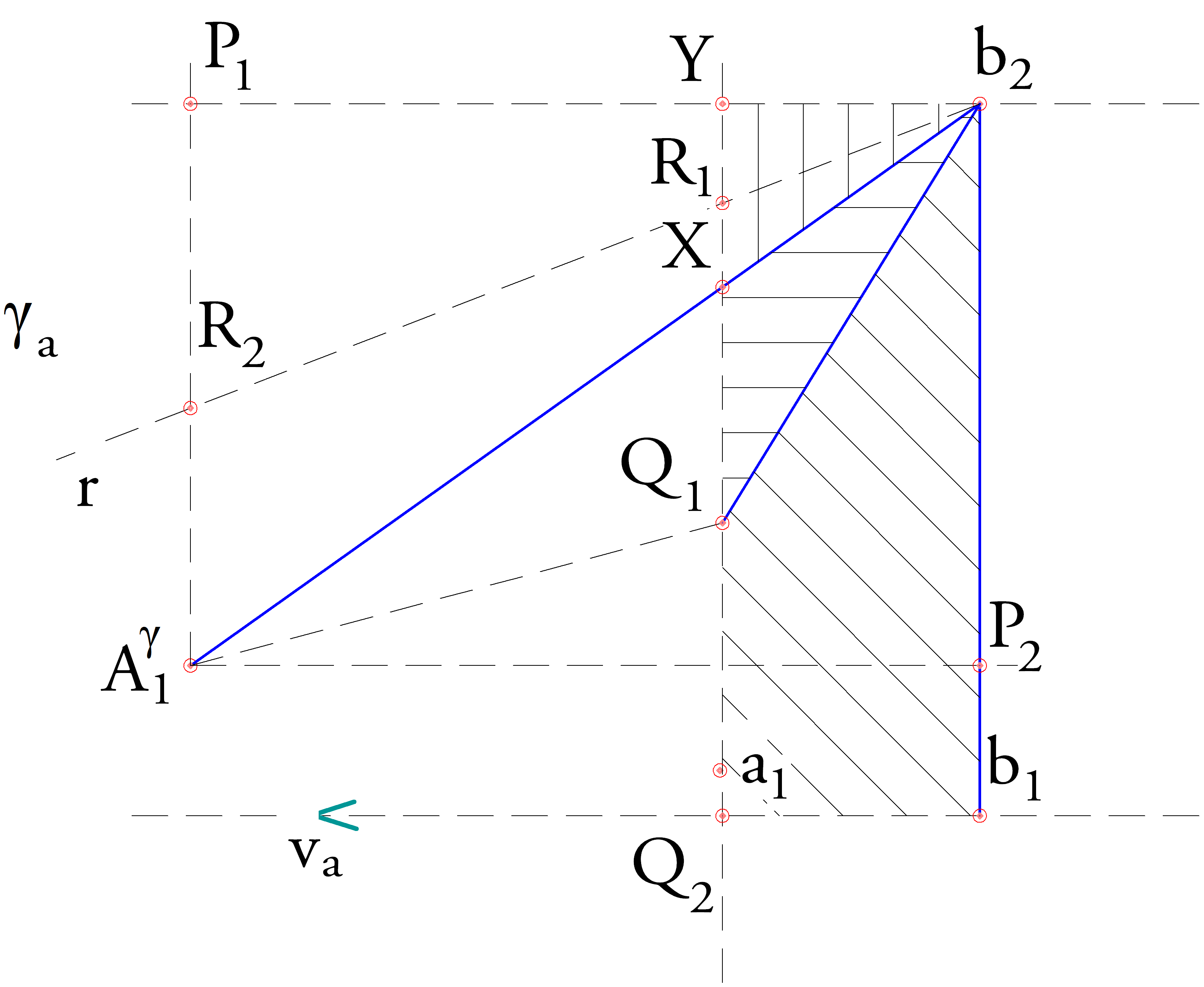}\label{fig2}  \end{center}\begin{center}  Figure 2 \end{center}

We have three differently hatched regions, in which the point $b_3$ may be projected (see Figure 2). These are two triangles $YXb_2$ and $Xb_2Q_1$, and a trapezium $Q_1b_2b_1Q_2$. Here point $Q_1$ is chosen in such a way that $\|Y-Q_1\|=\epsilon^{1/2}$. Note that the points $X, Q_1$ may appear on the segment $YQ_2$ in another order, but it does not change the proof, which is clear from the argument below.

If $b_3$ projects into the quadrangle, then
$$\tan \angle b^{\gamma}_3b_2b_1 = \frac{\langle u_{23}, v_a\rangle}{\langle u_{23}, u_{21}\rangle}\le \tan Q_1b_2b_1 = \frac{\langle Q_1-b_2, v_a\rangle}{\langle Q_1-b_2, u_{21}\rangle} = O\left(\frac{\epsilon^{2/3}}{\epsilon^{1/2}}\right) = O(\epsilon^{1/6}).$$

The inequality above holds since $b^{\gamma}_3$ falls in the angular domain $Q_1b_2b_1$, so $\angle b^{\gamma}_3b_2b_1 \le \angle Q_1b_2b_1$. The pre-last equality holds since $\|Q_2-b_1\|\le \|a_1-b_1\| \le \epsilon \|A_1-b_1\| = O(\epsilon^{2/3}). $ Thus, the vector $b_3^{\gamma}-b_2$  is $\epsilon$-parallel to $b_1-b_2$, which means that $b_3^{\gamma}-b_2$ and, consequently, $u_{23}$ is $\epsilon$-orthogonal to $v_a$ and to $u_{11}$.

For the cases when $b_3^{\gamma}$ is in one of the triangles we need the following lemma:

\begin{lem}\label{lem1} Consider a circle $S$ on the plane with the center in $O$ and with a point $B$ on the circle (See Figure 3). Take a point $A$ inside the circle, such that $AB$ is $\epsilon$-parallel to $OB$ and $\|A-B\| = o(\|O-B\|)$. Draw a line $l$ through $A$ that is orthogonal to $AB$. Then we have $\|A-B\| = o(\min\{\|P_1-A\|,\|P_2-A\|)$, where $P_1,P_2$ are the intersection points of $l$ and $S$.
\end{lem}

\begin{proof}Since the triangle $ABP_1$ is similar to $AP_2B'$, we have
\begin{multline*}\frac{\|A-B\|}{\|A-P_1\|} = \frac{\|A-P_2\|}{\|A-B'\|} = \frac{\|A-P_2\|}{2\|O-B\|\cos\angle OBA-\|A-B\|}<\\
\frac{\|P_1-P_2\|}{2\|O-B\|\cos\angle OBA-\|A-B\|}\to 0\ \ \  \text{ as } \ \frac{\|A-B\|}{\|O-B\|}\to 0.\end{multline*}
Similar reasoning applies to $\|A-P_2\|$.
\end{proof}

\begin{center}  \includegraphics[width=80mm]{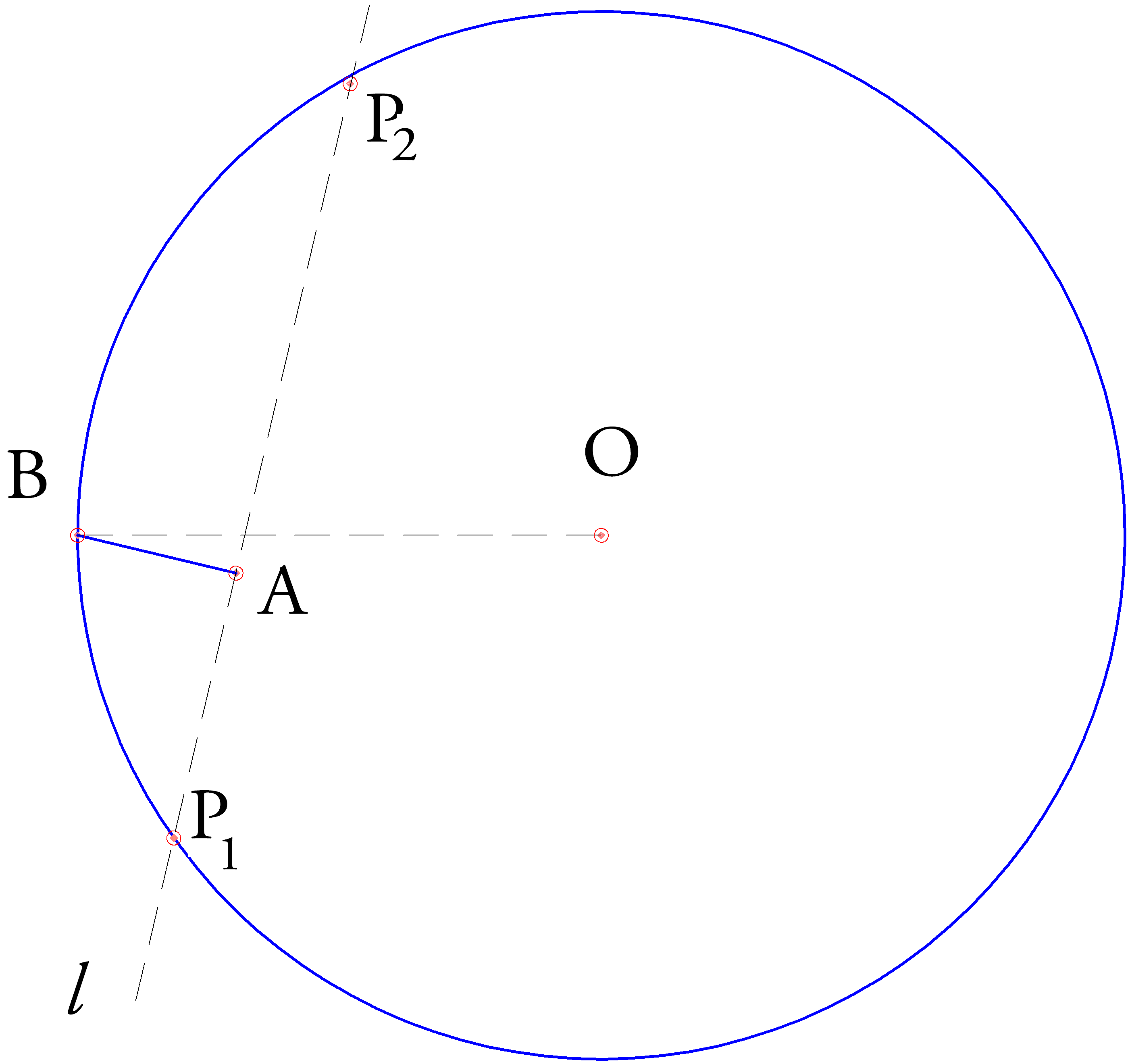}\label{fig3}  \end{center}\begin{center}  Figure 3 \end{center}

The following fact we use several times, so we formulate it as an observation:
\begin{obs}\label{obs2} If $\|b_3-b_3^{\gamma}\|\gg \|b_3^{\gamma}-b_2\|$, then $u_{23}$ is $\epsilon$-orthogonal to $u_{11}$.
\end{obs}

In the remaining triangular cases, according to Observation \ref{obs2}, it is enough to verify the assumption of the observation. Assume that $b_3^{\gamma}$ falls in the triangle $Yb_2Q_1$ and it lies on the line $r$ that passes through $b_2$. In the space $\R^d$ take a two-dimensional plane $\pi_r$ that contains $r$ and $b_3$. It is orthogonal to the plane $\gamma_a$. Denote by $S$ the (full-dimensional) sphere in $\R^d$ that has $A_1b_2$ as a diameter. Finally, consider a circle $S_r = \pi_r\cap S$ in the plane $\pi_r$. Denote its center by $O$ and the projection of $O$ on the plane $\gamma_a$ by $O^{\gamma}$.

By the choice of $A_1,a_1,b_1$ we have  $\|O-O^{\gamma}\|\le \|A_1-A_1^{\gamma}\|= O(\epsilon \|A_1-b_1\|).$ Denote by $P_1$ and $P_2$ (from Figure 2) are the projections of $A_1$ (or $A_1^{\gamma}$) on $Yb_2$ and $b_1b_2$ respectively. At the same time,
\begin{equation}\label{eq14}\|O-b_2\|\ge \frac 12\begin{cases} \|P_1-b_2\| \ \ \ \text{if }b_3^{\gamma}\in XYb_2,\\ \|P_2-b_2\|\ \ \ \text{if }b_3^{\gamma}\in XQ_1b_2\end{cases}\!\!\!=\Omega\bigl(\min\bigl\{\|A_1-b_1\|,1\bigr\}\bigr).\end{equation}
The inequality above holds due to the following. All the points $x$ of the segments $A_1^{\gamma}P_1, A_1^{\gamma}P_2$ lie inside the sphere $S$ since the angle $A_1xb_2$ is not acute. Thus, the intersection of $r$ and these segments lie inside the circle $S_r$ and the distance from the intersection point to $b_2$ is less than the diameter of $S_r$. On the other hand, any point on $A_1^{\gamma}P_i$ is farther away from $b_2$ than $P_i$. The equality in (\ref{eq14}) holds since $\|P_2-b_1\| = O(\epsilon \|A_1-b_1\|) = O(\epsilon^{2/3})$.

Therefore, $$\frac{\|O-O^{\gamma}\|}{\|O-b_2\|}\le \frac{O(\epsilon \|A_1-b_1\|)}{\Omega(\min\{\|A_1-b_1\|,1\}}= O(\epsilon^{2/3}).$$
It means that $Ob_2$ and $O^{\gamma}b_2$ are $\epsilon$-parallel, and, consequently, $Ob_2$ and $b^{\gamma}_3b_2$ are $\epsilon$-parallel.

Since $b_3$ and $b_2$ form a double-normal pair, $\angle b_2b_3A_1\le \pi/2$, which means that point $b_3$ lies on the outside of the (open) ball bounded by the sphere $S$, or, in other words, in the plane $\pi_r$ the point $b_3$ on the outside of the circle $S_r$. We also note that both $P_1$ and $P_2$ lie on $S$.

Assume for a second that we proved that $\|b_3^{\gamma}-b_2\| = o(\|O-b_2\|)$. Then we apply Lemma \ref{lem1} and obtain that $\|b_3-b_3^{\gamma}\|\gg \|b_3^{\gamma}-b_2\|$, which is the desired statement. So it is enough to show for both triangles $XYb_2$ and $Xb_2Q_1$ that the radius of $S_r$ is ``much bigger'' than $\|b_3^{\gamma}-b_2\|$.

Consider first $XYb_2$. Suppose that $r$ intersects $XY$ and $A_1^{\gamma}P_1$ in the points $R_1$ and $R_2$ respectively.
Both $R_i$ are inside $S_r$, so the leftmost point of intersection of $S_r$ and $r$ lies to the left from $A_1^{\gamma}P_1$ (or coincides with one of $P_1, A_1^{\gamma}$), thus, $2\|O-b_2\|\ge \|R_2-b_2\|$.
Therefore, $$\|b_3^{\gamma}-b_2\|\le \|R_1-b_2\|= O(\epsilon \|R_2-b_2\|)= o(\|R_2-b_2\|) = o(\|O-b_2\|),$$ where the first equality holds because of the choice of $A_1,a_1,b_1$.

Consider the case of the triangle $Xb_2Q_1$. On the one hand, in this case we have $$\|b_3^{\gamma}-b_2\| \le \|Y-Q_1\|+\|Y-b_2\| = O(\epsilon^{1/2}).$$ On the other hand, by (\ref{eq14}) we have $\|O-b_2\| = \Omega(1)$. This again gives $\|b_3^{\gamma}-b_2\| = o(\|O-b_2\|)$.\\

\textbf{Case 2} The segment $A_1b_1$ is ``much bigger'' than $b_2b_1$:  $\|A_{1}-b_{1}\| \gg \epsilon^{-1/3}.$

\begin{center}  \includegraphics[width=140mm]{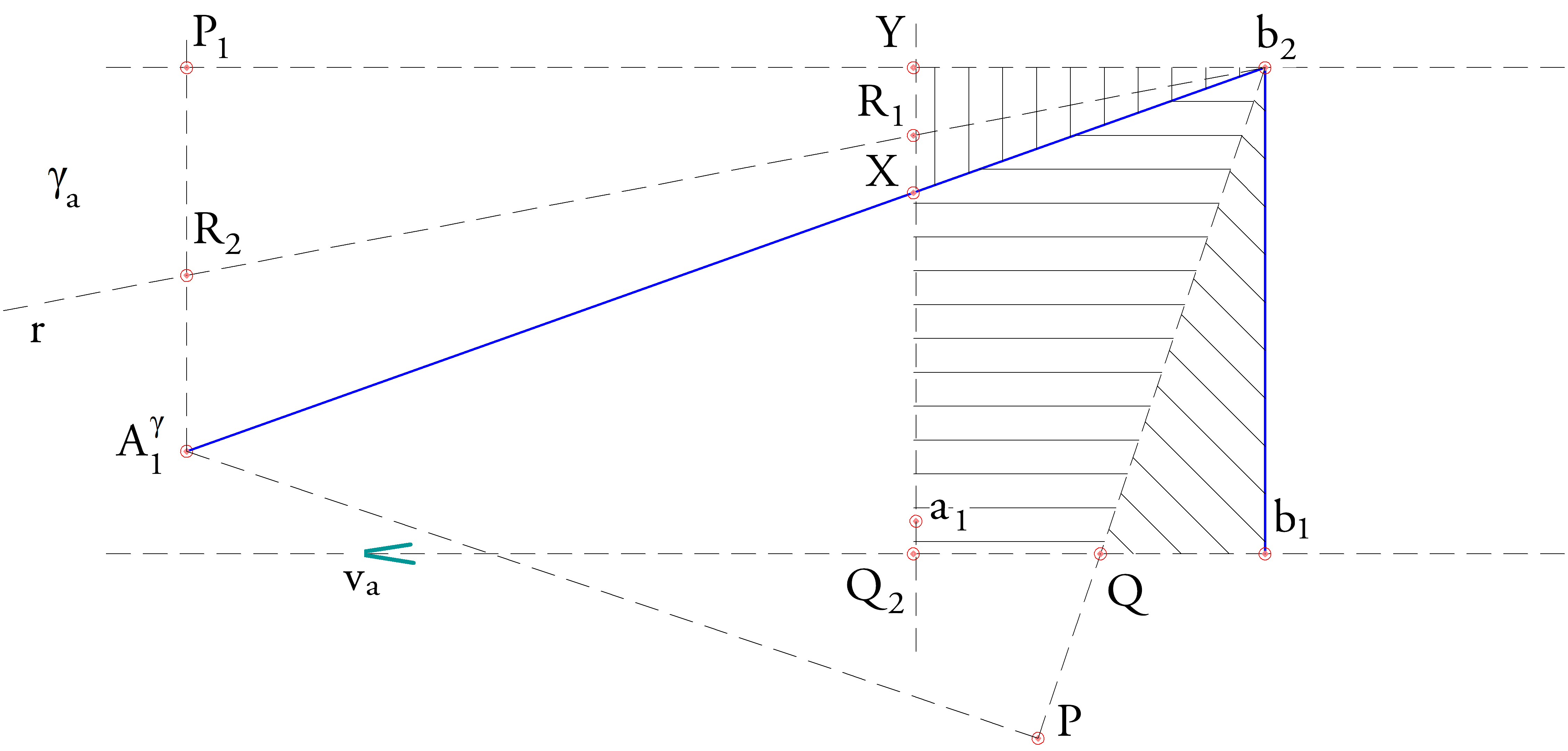}\label{fig4}  \end{center}\begin{center}  Figure 4 \end{center}

 This situation is depicted on Figure 4. The point $P$ from  Figure 4 is chosen in the following way. First, $\angle Qb_2b_1 = \epsilon^{1/6}$, and, second, $\angle A^{\gamma}_1Pb_2 = \pi/2$. The last condition is equivalent to saying that $P$ lies on the sphere with diameter $A_1b_2$. The point $Q$ is the intersection of the line $Pb_2$ and the segment $Q_2b_1$ or the segment $Q_2X$.

 In this case we again have three differently hatched regions, in which the point $b_3$ may be projected.  On the figure these are two triangles $YXb_2$ and $Qb_2b_1$, and a quadrangle $Xb_2QQ_2$. However, the line $b_2P$ may not intersect the segment $Q_2b_1$ and intersect $Q_2X$ instead.

 The argument for the situation when $b_3$ projects inside $XYb_2$ is identical to the one for that situation in Case 1.

 The vectors $P-b_2$ and $b_1-b_2$ are $\epsilon$-parallel, so if the point $b_3^{\gamma}$ lies in the angular domain bounded by $b_2Q$ and $b_2b_1$ (the triangle $Qb_2b_1$ in Figure 4), then $b_3^{\gamma}-b_2$ is $\epsilon$-parallel to $b_1-b_2$ and, thus, $u_{32}$ and $u_{11}$ are $\epsilon$-orthogonal.

 The last case to consider is when $b_3^{\gamma}$ falls inside the angular domain bounded by $A_1^{\gamma}b_2$ and $Pb_2$ (the horizontally hatched quadrangle on Figure 4). Arguing as in Case 1, consider the circle $S_r$. We again show that the
  radius $\|O-b_2\|$ of $S_r$ is ``much bigger'' than $\|b_3^{\gamma}-b_2\|$. On the one hand, we have $$\|b_3^{\gamma}-b_2\|\le \|b_1-b_2\|+ \|Y-b_2\| = O(1+\epsilon \|A_1-b_1\|).$$ On the other hand, using the same argument as the one used for the inequality in (\ref{eq14}) we see that the diameter of $S_r$ is at least $\|P-b_2\|,$ and $$\|P-b_2\|\ge \sin\angle Pb_2b_1\|A_1^{\gamma}-b_2\| = \Omega (\epsilon^{1/6}\|A_1-b_2\|)= \Omega (\epsilon^{1/6}\|A_1-b_1\|).$$
  Thus, $\|O-b_2\| = \Omega (\epsilon^{1/6}\|A_1-b_1\|).$ Combined with the condition that $\|A_{1}-b_{1}\| \gg \epsilon^{-1/3}$ we get that the $\|b_3^{\gamma}-b_{2}\| = o(\|O-b_2\|).$ At the same time, as in Case 1 we have $\|O-O^{\gamma}\| = O(\epsilon \|A_1-b_1\|)$ and therefore
 $$\frac{\|O-O^{\gamma}\|}{\|O-b_2\|}= O(\epsilon^{5/6}).$$
 Therefore,  $Ob_2$ and its projection on the plane $\gamma_a$ are $\epsilon$-parallel. Thus, we can apply Lemma \ref{lem1} and get that $\|b_3-b_3^{\gamma}\|\gg \|b_3^{\gamma}-b_2\|$. By Observation \ref{obs2} this concludes the proof. \\

\textbf{Step 3. } We set $\epsilon=1/n$ and choose $A_i^{(n)}\!,a_i^{(n)}\!,b_i^{(n)}\!,c_i^{(n)}$ as in the first part of the proof.
All vectors $u_{ij}^{(n)}$ lie on a unit sphere, which is a compact space. Thus, passing to a subsequence, we can assume that for each $i,j$ the sequence $u_{ij}^{(n)}$ converges to, say, $u_{ij}$. It is easy to see from  the properties (\ref{e8}), (\ref{e9}), (\ref{e10}) that $\langle u_{ii}, u_{jj}\rangle = 0$ and that $\langle u_{ii}, u_{jl}\rangle = 0$ for any $i, j, l = 1,\ldots, k$ if $j\neq l$. Thus, we have verified parts (\ref{e1}) and (\ref{e2}) of Theorem~\ref{thbase}.\\

\textbf{Step 4. } This is the last part of the proof, in which we are going to obtain the bound on the dimension of $\lin U_2$. To bound $\dim \lin U_2$, we want to work with points instead of vectors. All $u_{ij}$ with $i\ne j$ come as a limit of scaled $b_j^{(n)}-b_i^{(n)}$. We choose some $\delta>0$ and $n$ sufficiently large so that the angle between $u_{ij}^{(n)}$ and $u_{ij}$ is less than $\delta/2$ for any $i\ne j$. Next, we project all the points $b_i^{(n)}$ on the subspace $\lin U_2$, which is interpreted as an affine plane passing through zero. Denote by $B_i^{(n)}$ the corresponding projections. They are obviously distinct for different points if $\delta$ is small enough. Otherwise, if $B_i^{(n)} = B_j^{(n)}$ for some $i\ne j$, then $u_{ij}^{(n)}$ is orthogonal to $U$ and, consequently,  to $u_{ij}$.

The angle between $B_j^{(n)}-B_i^{(n)}$ and $u_{ij}^{(n)}$ is less than the angle between $u_{ij}^{(n)}$ and $u_{ij}$ (since $B_j^{(n)}-B_i^{(n)}$ is a scaled projection of $u_{ij}^{(n)}$ on $U$). Therefore, by the triangle inequality, for any pairwise distinct $i,j,l=1,\ldots,k$ we have $\angle B_i^{(n)}B_j^{(n)}B_l^{(n)}\le \angle b_i^{(n)}b_j^{(n)}b_l^{(n)} + \alpha_1+\alpha_2\le \pi/2+\alpha_1+\alpha_2,$ where $\alpha_1$ is the angle between vectors $u_{ji}^{(n)}$ and $B_i^{(n)}-B_j^{(n)}$ and $\alpha_2$ is the angle between vectors $u_{jl}^{(n)}$ and $B_l^{(n)}-B_j^{(n)}$. Both $\alpha_1$ and $\alpha_2$ are less than $\delta/2$, so for any any pairwise distinct $i,j,l=1,\ldots,k$ we have $\angle B_i^{(n)}B_j^{(n)}B_l^{(n)}\le \pi/2 +\delta$.

The above considerations allowed us to construct an affine plane of dimension  $\dim \lin U_2$ and a set of $k$ points on this plane without triples that form an angle bigger than $\pi/2+\delta$ for some $\delta>o$. Moreover, we may make $\delta$ as small a positive number as we want it to be. It turns out that $k\le 2^{\dim \lin U_2}$, which is stated in

\begin{lem}\label{lem2} For any natural $d$ there exists a $\delta>0$ such that the following holds.  Consider any $k$ points $p_1,\ldots, p_k$ in $\R^d$ such that for any distinct $i,j,l\in \{1,\ldots, k\}$ we have $\angle p_ip_jp_k\le \pi/2 +\delta$. Then $k\le 2^{d}$.\\
\end{lem}

\textbf{Remark.} The statement of Lemma \ref{lem2} for $\delta =0$ is a result by L. Danzer and B. Gr\" unbaum \cite{DG}.  A statement equivalent to Lemma \ref{lem2} is mentioned without a proof in \cite{EF}, so it may be well known. However, in \cite{EF} it looks like the authors somewhat inaccurately attributed Lemma \ref{lem2} to Danzer and Gr\" unbaum. Thus, we present the proof of the lemma here for completeness. We remark that, contrary to intuition, it does not follow from the result due to Danzer and Gr\" unbaum via a direct compactness argument.


\begin{proof} We prove the statement by induction. It is easy to check for the plane. Now suppose we have proved it for the dimension $d-1$ and pass on to the dimension $d$. Assume that the lemma is not true in $\R^d$. Then there exists a sequence of point sets $W_j=\{p_1^j,\ldots,p_k^j\}$ in $\R^d$ so that $\|W_j\| = k = 2^{d}+1$ and the maximal angle in $W_j$ is less than $\pi/2+1/j$. Scale and translate these sets so that all $W_j$ have unit diameter and so that all of them lie in a unit ball. Passing to subsequences, we have convergent sequence, which gives a limiting multiset of points $W = \{p_1,\ldots, p_k\}$. Naturally, in $W$ all the angles are non-obtuse and $W$ is of unit diameter. If none of $p_i$ coincide, then we get a contradiction with the result by Danzer and Gr\" unbaum. Thus, at least two points among $p_i$ must coincide.

Assume that the maximum multiplicity of an element in $W$  is $m$, while the corresponding point in $W$ is, w.l.o.g., $p_1$. Denote the  subset of points that participate in the sequences that do converge to $p_1$ in every $W_j$ by $M_j$. Since $W$ is of unit diameter, all of the points in $W$ cannot coincide and, thus, $m\le k-1 = 2^d$. Choose an integer $l$ so that $2^{l-1}<m\le 2^l$, $1\le l\le d$. Consider the set of  vectors $U_j$: $U_j = \{u_{sr}^j = p_r^j-p_s^j: p_r^j,p_s^j\in M_j, r\neq j\}$ and a limiting set of vectors $U$.  We claim that, first, $\dim \lin U\ge l$ and, second, $\aff W$ is orthogonal to $\lin U$ in the sense that any vector formed by points from $\aff W$ is orthogonal to any vector from $\lin U$.

To show that $\dim \lin U\ge l$ we have to apply induction. We choose sufficiently large $j$ so that the vectors from $U_j$ are close to the limiting vectors and, as in Step 4, project all points from $M_j$ on $\lin U$. We get angles as close to $\pi/2$ as we want to (it depends on $j$ only), so for a right choice of $j$ we can apply induction to $M_j$ and $\lin U$ and get the stated inequality.

To show that $W$ spans a plane that is orthogonal to $\lin U$, we assume the contrary. Choose a point from $W$, say, $p_2$, such that the vector $p_2-p_1$ is not orthogonal to $\lin U$. This means that there is a pair of indices $s, r$ and $\epsilon>0$, so that $p_2-p_1$ and $u_{sr}$ form an angle at least $\pi/2+4\epsilon$ big. If we choose a sufficiently large $j$, then $p_2-p_s$ and $p_2^j-p_s^j$ form an angle $\alpha_1$ at most $\epsilon$, similarly, $p_r-p_s$ and $p_r^j-p_s^j$ form an angle $\alpha_2$ at most $\epsilon$ and, finally, there is no angle bigger than $\pi/2+\epsilon$ in $W_j$. But then we arrive at a contradiction, since on the one hand, $\angle p_r^jp_s^jp_2^j\le \pi/2+\epsilon$, on the other hand, $\angle p_r^jp_s^jp_2^j\ge\angle p_rp_sp_2-\alpha_1-\alpha_2\ge \pi/2+2\epsilon$.

Therefore, $W$ has at least $\frac{2^d}m +1\ge 2^{d-l}+1$ distinct points and $\dim \aff W \le d-\dim \lin U\le d-l$. We may apply induction and arrive to a contradiction.
\end{proof}
This proves (\ref{e3}) and completes the proof of Theorem \ref{thbase}.\\

 \textbf{Remark. } As we have promised, we discuss here the differences between Theorem \ref{thbase} and Proposition 8 from \cite{PS1}.

 In the formulation (and in essence) the main difference is that in Theorem \ref{thbase} we have a much more general condition (\ref{e2}) rather than the condition ``$u_{ii}$ and $u_{ij}$ are orthogonal, if $i\ne j$'' from Proposition 8. We  also have the condition (\ref{e3}) that bounds the dimension of $\lin U_2$. This, as well as Step 4, is not present in the proposition.

 Step 1 in our proof is really similar to Step 1 from the proposition. A large part is just identical to the analogous part from the proposition. However, we choose 4 points (with specific distances) instead of 3 in the proposition, so the last paragraph of Step 1 is new.

 The parts of Step 2 that are presented (recall that we omit the proof of (\ref{e8}) and (\ref{e9})) are completely new.

 Step 3 is really similar to that step in the proposition, however, it is shorter, since we are working only with vectors.

 As we have already mentioned, there is no analogue of Step 4 in the paper \cite{PS1}.

\section{Lower bound on $\mathbf{k'(d)}$}
The proof of (\ref{e02}) is based on a slight modification of the proof of Theorem~\ref{thps2}. In this section we first give a complete proof of (\ref{e02}) and then the sketch of the original proof of Pach and Swanepoel. In this way the reader would have the chance to compare the two proofs. We split the section into two subsections.

\subsection{Proof of (\ref{e02})}

\noindent\textbf{Step 1. } Fix $p_1\ldots, p_m\in \R^d$ so that they form only acute angles. Thus, they are in convex position. For each $p_i$ choose a hyperplane $H_i$ that touches the convex hull of $\{p_1\ldots,p_m\}$ only in $p_i$ and a normal vector $u_i$ to $H_i$, that points towards the convex hull. Embed $\R^d$ as a plane $\Gamma$ in $\R^{d+m}$. For each $p_i$ consider a two-dimensional plane $\Pi_i$ spanned by $u_i$ and $v_i$, where $v_i$ are orthogonal to $\Gamma$ and are mutually orthogonal as well.\\

\noindent\textbf{Step 2. } For every $i$ do the following. Project all $p_j$ on $\Pi_i$ (they all fall on a half-line that starts at $p_i$ and goes in the direction of $u_i$). Choose point $c_i$ on the half-line so that the circle $C_i$ in $\Pi_i$ with center in $c_i$ and that passes through $p_i$ contains strictly inside all the projections of $p_j$. Denote the other intersection of the halfline and the circle by $q_i$ (See Fig. 5, where we omitted the subscript indices $i$ for simplicity). Choose an $\epsilon$ and a segment $I_i$ of the halfline in such a way that $I_i$ is disjoint with the $\epsilon$-neighborhoods of the points $p_i,q_i$, but contains each $p'_j$ with an $\epsilon$-neighborhood (on the halfline). The endpoints of $I_i$ are denoted by $a$ and $b$.



\begin{center}  \includegraphics[width=100mm]{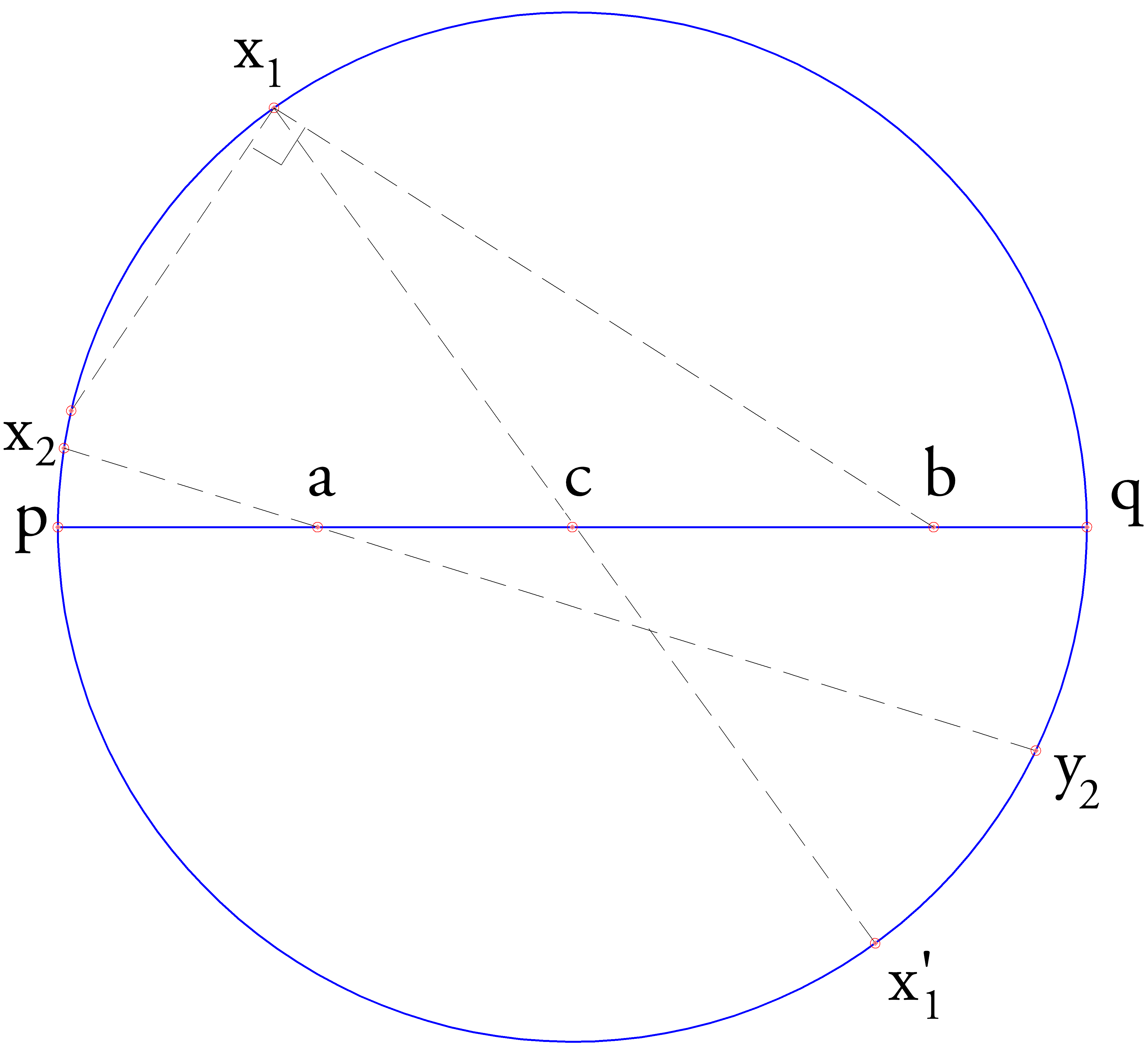}\label{fig5}  \end{center}\begin{center}  Figure 5 \end{center}

\noindent\textbf{Step 3. } In this step we construct the vertices of the (supergraph of a) complete multipartite graph $K_m(N)$. Choose points $x_1^i,\ldots, x_N^i\in C_i$ in the following way. First, for any $z$ that belongs to the segment $I_i$ and for any $s,l\in \{1,\ldots, N\}, s\ne l,$ we have $\angle zx_sx_l< \pi/2$. Second, all $x_s$ are in $\epsilon$-neighborhood of $p_i$. Put $V_i = \{x_1^i,\ldots, x_N^i\}$.

In what follows we justify that such a choice is possible. For simplicity, for the rest of Step 3 we drop all the subscripts (and superscripts) $i$ that refer to the part $V_i$. The second condition is very easy to satisfy, thus, we only have to show how to choose $x_s$ such that for any $z\in I$ and for any different $s,l=1,\ldots, N$ we have $\angle zx_sx_l< \pi/2$. Recall that the endpoints of $I$ are denoted by $a,b$ (see Figure 5). We describe the procedure of choosing $x_s$ recursively. We take as $x_1$ an arbitrary point from the upper semicircle that is $\epsilon$-close to $p$ (but, obviously, different from $p$). Next, having chosen $x_1,\ldots, x_s,$ we choose $x_{s+1}$ from the arc $px_s$ so that all the angular constraints are satisfied. It is easy to see that for any $0\le k\le s-1$ and any point $z$ between $a$ and $b$ we have $\angle x_{s+1}x_{s-k}z\le \angle x_{s+1}x_{s}z$ and $\angle x_{s-k}x_{s+1}z\le \angle x_{s},x_{s+1},z$. Moreover, $\angle x_{s}x_{s+1}z\le \angle x_{s}x_{s+1}a$ and $\angle x_{s+1}x_{s}z\le \angle x_{s+1}x_{s}b$. Therefore, we need to satisfy only two constraints: $\angle x_{s+1}x_sb$ is acute, and $\angle x_sx_{s+1}a$ is acute.

To satisfy the first constraint, we draw a chord from $x_s$ that is orthogonal to $bx_s$ and choose $x_{s+1}$ from an open arc connecting $p$ and the second intersection point of the chord.

Next we explain how to satisfy the second constraint. Assume we made a choice of $x_{s+1}$. Consider a point $x'_s\in C$, which is symmetric to $x_s$ with respect to the center $c$. Next, consider a chord that contains $x_{s+1}$ and $a$. Its other intersection with $C$ denote by $y_{s+1}$. If $y_{s+1}$ lies on an arc  $x_s'q$, then $\angle x_sx_{s+1}a$ is acute. Thus, it suffices to choose $x_{s+1}$ on the open arc $a_sp$, where $a_s$ is the second intersection point of the chord that passes through $x_s'$ and $a$.

It is clear that we can satisfy both restrictions if we choose $x_{s+1}$ close enough to $p$.\\

\noindent\textbf{Step 4. } In this step we verify that the strict double-normal graph on the set of vertices  $\cup_i V_i$ contains $K_m(N)$ as a subgraph. We need to show that any two points $y,z$ from different $V_i$ form a strict double normal. For that we need to do a case analysis. We assume throughout that $\epsilon$ is chosen to be sufficiently small. Take $x,y,z\in \cup_i V_i$. If $x,y,z$ are all from different parts $V_i,V_j,V_k$, then $\angle xyz$ is acute since points $x,y,z$ are $\epsilon$-close to $p_i,p_j,p_k$ and $\angle p_ip_jp_k$ is acute. If $x,z\in V_i$, $y\in V_j$, then $\|x-z\|$ is small compared to $\|x-y\|,\|z-y\|$, so the angle $\angle xyz$ is small and, thus, acute. The last case is if $x,y\in V_i$ and $z\in V_j$. Then $\angle xyz$ is acute since the angle $\angle xyz'$ is acute, where $z'$ is a projection of $z$ on $\Pi_i$. In turn, since $z'$ is in the $\epsilon$-neighborhood of $p'_j$, we have $z'\in I_i$. Therefore, Step 3 tells us that $\angle xyz'$ is acute. Thus, $\cup_i V_i$ indeed contains $K_m(N)$ as a subgraph.

\subsection{Sketch of the proof of Theorem \ref{thps2}}

\textbf{Step 1. } Fix $p_1\ldots, p_m\in \R^d$ and $u_1,\ldots, u_m\in
\R^d$ as in the assumption of the theorem. Embed $\R^d$ as a plane $\Gamma$ in $\R^{d+m}$. For each $p_i$ consider the two-dimensional plane $\Pi_i$ spanned by $u_i$ and $v_i$, where $v_i$ are orthogonal to $\Gamma$ and are mutually orthogonal as well.\\

\noindent\textbf{Step 2. } For every $i$ do the following. Project all $p_j$ on $\Pi_i$ (they all fall on a half-line that starts at $p_i$ and goes in the direction of $u_i$). Choose point $c_i$ on the half-line so that the circle $C_i$ in $\Pi_i$ with center in $c_i$ and that passes through $p_i$ has all the projections of $p_j$ inside an open segment  $c_iq_i$, where $q_i$ is the intersection of the halfline and the circle, which is different from $p_i$. By $I_i$ denote the segment $c_iq_i$, out of which we exclude $\epsilon$-neighborhoods of $c_i,q_i$ for some small $\epsilon>0$. Make sure that $\epsilon$ is so small that $p'_i$ with their $\epsilon$-neighborhood fall inside $I_i$.\\

\noindent\textbf{Step 3. } Choose points $x_1^i,\ldots, x_N^i\in C_i$ in the following way. First, for any $z$ that belongs to the segment $I_i$ and for any different $j,l=1,\ldots, N$ we have $\angle zx_jx_l< \pi/2$. Second, all $x_i$ are in $\epsilon$-neighborhood of $p_i$. Put $V_i = \{x_1^i,\ldots, x_N^i\}$.\\

\noindent\textbf{Step 4. } This step repeats word for word Step 4 from the previous section (or, rather, the way around: Step 4 in the proof of (\ref{e02}) is the same as Step 4 in this proof).\\

\textbf{Remark. } The modifications that lead to the proof of the lower bound in Theorem \ref{thmain}, actually simplify significantly the proof of the lower bound as stated in Theorem \ref{thps}. It is due to the fact that we may use an Erd\H os-F\"uredi-type bound right away, while in \cite{PS1} the authors had to consider and analyze an analogue of the Erd\H os-F\" uredi construction that satisfies an additional restriction on scalar products.

\section{Acknowledgements} The author is very grateful to J\'anos Pach, Sasha Polyanskii and Konrad Swanepoel for several stimulating discussions on the subject and to one of the reviewers, who had very carefully read the manuscript and pointed out several problems with the proofs and the exposition. It helped to improve a lot the presentation of the paper.

\end{document}